\documentclass[a4paper,11pt]{amsart}
\usepackage{amsfonts,amsthm,amsmath,amssymb,graphicx}
\theoremstyle{plain}
\newtheorem{lem}{Lemma}[section]
\newtheorem{prop}[lem]{Proposition}
\newtheorem{thm}[lem]{Theorem}

\theoremstyle{definition}

\theoremstyle{remark}

\numberwithin{equation}{section}

\newcommand{\N}{\mathbb N}

\renewcommand{\epsilon}{\varepsilon}

 \usepackage[usenames,dvipsnames]{pstricks}
 \usepackage{epsfig}
 \usepackage{pst-grad} 
 \usepackage{pst-plot} 

\title[Fast Khintchine spectra]{Upper and lower fast Khintchine spectra in continued fractions}
\date{}

\author{Lingmin Liao}
\address{Lingmin Liao\\LAMA UMR 8050, CNRS,
Universit\'e Paris-Est Cr\'eteil, 61 Avenue du
G\'en\'eral de Gaulle, 94010 Cr\'eteil Cedex, France}
\email{lingmin.liao@u-pec.fr}
\author{Micha\l\ Rams}
\address{Micha\l\ Rams\\Institute of Mathematics\\ Polish
Academy of Sciences\\ ul.
\'Sniadeckich 8, 00-656 Warszawa\\ Poland }
\email{rams@impan.pl}
\begin{document}
\begin{abstract}
For an irrational number $x\in [0,1)$, let $x=[a_1(x), a_2(x),\cdots]$ be its continued fraction expansion. Let $\psi : \mathbb{N} \rightarrow \mathbb{N}$ be a function with $\psi(n)/n\to \infty$ as $n\to
\infty$. The (upper, lower) fast Khintchine spectrum for $\psi$ is defined as the Hausdorff dimension of the set of numbers $x\in (0,1)$ for which the (upper, lower) limit of $\frac{1}{\psi(n)}\sum_{j=1}^n\log a_j(x)$ is equal to $1$. The fast Khintchine spectrum was determined by Fan, Liao, Wang, and Wu.
We calculate the upper and lower fast Khintchine spectra. These three spectra can be different.
\end{abstract}

\maketitle
\def\thefootnote{}
\def\thefootnote{\arabic{footnote}}

\section{Introduction}
Each irrational
number $x\in [0,1)$ admits a unique infinite continued fraction
expansion of the form
\begin{eqnarray}\label{ff1}
x=\frac{\displaystyle 1}{\displaystyle a_1(x)+ \frac{\displaystyle
1}{\displaystyle a_2(x)+\frac{\displaystyle 1}{\displaystyle
a_3(x)+\ddots}}},
\end{eqnarray}
where the positive integers $a_n(x)$, called the partial quotients of $x$, can be generated by using the Gauss
transformation $T:[0,1)\to [0,1)$ defined by  $$ T(0):=0, \ \
\text{and} \ \ T(x):=\frac{1}{x} \ {\rm{(mod \ 1)}}, \ {\rm{for}}\ x\in (0,1).
$$
In fact, let $a_1(x)= \lfloor x^{-1}\rfloor$ ($\lfloor \cdot \rfloor $ stands for the integer part), then $a_n(x)=a_1(T^{n-1}(x))$ for $n\ge 2$. For simplicity,  (\ref{ff1}) is often written as $x=[a_1,a_2,\cdots]$.

For any irrational number $x\in (0,1)$,  its {\it Khintchine exponent} is defined by the limit (if it exists)
\begin{align}\label{def-Khin}
\xi(x):=\lim_{n\to\infty}\frac{\log
a_1(x)+\cdots+\log a_n(x)}{n}.
\end{align}
Khintchine \cite{Kh35} (see also \cite[p.\,93]{Kh}) proved that for Lebesgue almost all points $x$, we have
\[
\xi(x)=\int_0^1 \frac {\log a_1(x)} {(1+x)\log 2} dx= \log (2.6854...) =:\xi_0.
\]
Though Khintchine did not use ergodic theory in his orignal proof, his result is a consequence of Birkhoff Ergodic Theorem, by the fact that the Gauss transformation $T$ is ergodic with respect to the measure
$dx/ ((1+x)\log 2)$ and that the average in (\ref{def-Khin}) can be written as a Birkhoff ergodic average.

From a multifractal point of view, one is also interested in the sets of points with a given Khintchine exponent which is different from $\xi_0$. The Hausdorff dimension (denoted by $\dim_H$) of the level sets
$$
E(\alpha):=\left\{x\in [0,1): \lim_{n\to\infty}\frac{\log
a_1(x)+\cdots+\log a_n(x)}{n}=\alpha\right\}, \quad \alpha>0,
$$ was calculated in  \cite{FLWW}. It turns out that the {\it Khintchine spectrum}, i.e., the function $\alpha \mapsto \dim_H E(\alpha)$ is a real-analytic curve increasing on $[0, \xi_0]$ and decreasing on $(\xi_0, \infty)$. Further, the Hausdorff dimension of $E(+\infty)$ is equal to $1/2$, which means that there are lots of numbers with infinite Khintchine exponent. This thus leads to the question of detailed classification of numbers with infinite Khintchine exponent.

Let $\alpha>0$ and let $\psi:\mathbb{N}\to \mathbb{N}$ such that $\psi(n)/n\to \infty$ as $n\to\infty$. We consider the following subsets of $E(+\infty)$:
$$
E(\psi, \alpha):=\left\{x\in [0,1): \lim_{n\to\infty}\frac{\log
a_1(x)+\cdots+\log a_n(x)}{\psi(n)}=\alpha\right\}  \quad \alpha>0.
$$ 
The so-called {\it fast Khintchine spectrum}, i.e., the Hausdorff dimension of $E(\psi, 1)$ was obtained in \cite{FLWW13}. Let $\psi$ and $\tilde{\psi}$ be two functions defined on
$\N$. We say $\psi$ and $\tilde{\psi}$ are {\em equivalent} if $\frac{\psi(n)}{\tilde{\psi}(n)}\to 1$ as $n\to \infty$.
Define
\begin{align}\label{def-beta}
\beta=\beta(\psi):=\limsup_{n\to
\infty}\frac{\psi(n+1)}{\psi(n)}.
\end{align} 
The authors of \cite{FLWW13} proved the following theorem.
\begin{thm}[\cite{FLWW13}, Theorem 1.1]\label{th-FLWW}  Let $\psi:\N\to \N$ with ${\psi(n)}/{n}\to \infty$ as $n\to \infty$.
If $\psi$ is equivalent to a nondecreasing function, then $E(\psi, 1)\ne
\emptyset$ and
$$ \dim_HE(\psi, 1)=\frac{1}{1+\beta(\psi)}.
 $$ Otherwise,
 $E(\psi, 1)=\emptyset$.
\end{thm}
We remark that the level $\alpha=1$ of the level set $E(\psi, 1)$ in Theorem \ref{th-FLWW} can be replaced by any level $\alpha>0$. In fact, observing that for all $\alpha>0$, $E(\psi, \alpha)=E(\widehat{\psi}, 1)$ with $\widehat{\psi}=\alpha^{-1}\cdot \psi$ and that $\beta(\widehat{\psi})=\beta(\psi)$, we obtain
\[
 \dim_HE(\psi, \alpha)= \dim_HE(\widehat{\psi}, 1)=\frac{1}{1+\beta(\widehat{\psi})}=\frac{1}{1+\beta(\psi)}, \quad \forall \alpha>0.
\]
Similarly, when $\psi$ is not equivalent to a nondecreasing function, we have $E(\psi, \alpha)=\emptyset$ for all $\alpha>0$.

\medskip

In this note, we consider the following sets
\begin{align}\label{upper-set}
\overline{E}(\psi)=\left\{x\in [0,1]: \limsup_{n\to\infty}\frac{\log a_1(x)+\cdots+\log a_n(x)}{\psi(n)}=1\right\},
\end{align}
and
\begin{align}\label{lower-set}
\underline{E}(\psi)=\left\{x\in [0,1]: \liminf_{n\to\infty}\frac{\log a_1(x)+\cdots+\log a_n(x)}{\psi(n)}=1\right\}.
\end{align}
Their Hausdorff dimensions are called {\it upper} and {\it lower fast Khintchine spectra}.

Remark that we only consider the level $\alpha=1$ here, since for other levels the Hausdorff dimension will not change, as in Theorem \ref{th-FLWW}.

Our main result is as follows.

\begin{thm}\label{main} Assume that $\psi:\mathbb{N}\to\mathbb{N}$
satisfies $\psi(n)/n\to\infty$ as $n\to\infty$. Write
\begin{align}\label{def-bB}
\liminf_{n\to\infty}\frac{\log \psi(n)}{n}=\log b \quad \text{and} \quad \limsup_{n\to\infty}\frac{\log \psi(n)}{n}=\log B.
\end{align}
Assume $b, B \in (1, \infty]$. Then 
$$
\dim_H\overline{E}(\psi)=\frac{1}{1+b}\quad \text{and} \quad \dim_H\underline{E}(\psi)=\frac{1}{1+B}.
$$
\end{thm}

We remark that $b\leq B\leq \beta$. The first inequality is directly from the definitions (\ref{def-bB}).  For the second, by (\ref{def-beta}), for any small $\epsilon>0$, there exists $n_0$ such that $\psi(i+1)/\psi(i) \leq \beta+\epsilon$ for all $i\geq n_0$. Then
\[
\psi(n) = \psi(n_0) \cdot \prod_{i=n_0}^{n-1} \frac{\psi(i+1)}{\psi(i)} \leq \psi(n_0) \cdot (\beta + \epsilon)^{n-n_0}.
\]
Hence the second inequality follows from  the definition (\ref{def-bB}) of $B$.
However, one can construct some $\psi$ such that the three values $ b$, $B$ and $\beta$ are all different.

 We also remark that from Theorem \ref{main}, the sets
$\overline{E}(\psi)$ and $\underline{E}(\psi)$ are always nonempty.

Our result can be considered as a contribution to the multifractal analysis of Birkhoff sums (averages) of dynamical systems.
In history, the first multifractal analysis of Birkhoff averages may be due to Besicovitch who studied the frequencies of digits in binary expansions. Let $x\in [0,1]$ and let $x=.x_1x_2\dots, $ with $x_i\in \{0,1\}$ denote its
binary expansion. Besicovitch (\cite[p.\,322]{Be}) obtained the Hausdorff dimension of the following level sets
\[
\left\{x\in [0,1] : \limsup_{n\to \infty} \frac{x_1+\cdots + x_n} n \leq \alpha\right\}, \ \alpha\in [0,1/2].
\]
Let $(X,d)$ be a metric space, $T$ be a (piecewise) continuous transformation on $X$, and $\phi : X\rightarrow \mathbb{R}$ be a real-valued (piecewise) continuous function. The Hausdorff dimension of the level sets of Birkhoff averages:
\[
\left\{x\in X : \lim_{n\to\infty} \frac{\phi(x)+\phi(Tx)+\cdots + \phi(T^nx)} {n} = \alpha\right\}, \quad \alpha\in \mathbb{R},
\]
were widely studied (\cite{Ol, FFW, PW}).
In \cite[Theorem 3.3]{FST}, a multifractal analysis result of the above level sets replacing $\lim$ by $\limsup$ and $\liminf$ was established for the full shift over two symbols and was applied to the study of dynamical Diophantine approximation.

In contrast to continued fractions, in symbolic dynamical systems of finitely many symbols there is no fast spectrum since the Birkhoff averages are usually bounded (for example, when $\phi$ is continuous). Thus the fast spectra studied in this note and in \cite{FLWW, WX, X, LR} are new subjects for continued fractions and can have generalization in symbolic dynamical systems of infinitely many symbols.

The result of \cite[Theorem 3.3]{FST} 
shows that in symbolic dynamical systems with finitely many symbols, the multifractal spectra remain the same when $\lim$ is replaced by $\limsup$ or $\liminf$.
However, our main result Theorem \ref{main} proves that in the case of fast Khintchine spectra, if we change $\lim$ to $\limsup$ or $\liminf$, the result changes essentially. This uncovers a new phenomenon in continued fractions and in symbolic dynamical systems of infinitely many symbols.

\bigskip
\section{Preliminaries}
%
 For any $n \geq 1$ and $(a_1,a_2,\cdots,a_n)\in
\mathbb{N}^n$, define $$ I_n(a_1, a_2, \cdots, a_n)=\big\{x\in
[0,1):\ a_1(x)=a_1, \cdots, a_n(x)=a_n\big\},
$$ which is the set of numbers starting with $(a_1,\cdots, a_n)$ in their continued fraction expansions,
and is called a {\em basic interval} of order $n$. The length of a basic interval will be denoted by $|I_n|$.
%
\begin{prop}[\cite{Kh}, p.\,66, p.\,68]\label{p2.1}
For any $n\geq 1$ and $(a_1,\cdots, a_n)\in \mathbb{N}^n$,
\begin{equation}\label{length-int}
  \left(2^n \prod_{k=1}^na_k\right)^{-2} \le |I_n(a_1,\cdots,a_n)|\le \left(\prod_{k=1}^na_k\right)^{-2}.
\end{equation}

\end{prop}

The following lemma is used to calculate the lower bound of the Hausdorff dimension of $\overline{E}(\psi)$.

Let $\ell\geq 2$ be some fixed real number and $\{s_n\}_{n\geq 1}$ be a sequence of real numbers such that $s_n\geq 1$. Set $$
F(\{s_n\}_{n=1}^{\infty};\ell):=\big\{x\in [0,1): s_n\leq a_n(x)<\ell
s_n, \ {\rm{for \ all}}\ n\geq 1\big\}.
$$
\begin{lem}[\cite{FLWW}, Lemma 3.2]\label{lemma-FLWW}
Under the assumption that $ s_n\to \infty$ as
$n\to \infty$, one has\begin{eqnarray*}
\dim_HF(\{s_n\}_{n=1}^{\infty};\ell)=\left(2+\limsup_{n\to \infty}\frac{\log s_{n+1}} {\log s_1s_2\cdots s_n}\right)^{-1}.
\end{eqnarray*}
\end{lem}

We remark that the Hausdorff dimension of the set $F(\{s_n\}_{n=1}^{\infty};\ell)$ does not depend on $\ell$. The original version of Lemma \ref{lemma-FLWW} asks for the sequence $\{s_n\}$ and the number $\ell$ to be positive integers. However, a slight modification of the proof also works for real numbers.

In fact, Lemma \ref{lemma-FLWW} has a more general form.  Let $s:=\{s_n\}_{n\geq 1}$ and $t:=\{t_n\}_{n\geq 1}$ be two sequences of
real numbers such that $s_n\geq 1, t_n >1$ for all $n\geq 1$.
Consider the following set
\[
F(s,t):=\big\{x\in [0,1): s_n\leq a_n(x)<
s_nt_n, \ {\rm{for \ all}}\ n\geq 1\big\}.
\]
Naturally, we do need to assume that for each $n$ there is an integer between $s_n$ and $s_n t_n$, otherwise $F(s,t)$ would be empty.

\begin{lem}\label{lemma-general}
Assume that $F(s,t)\neq \emptyset$, $ s_n\to \infty$ as
$n\to \infty$, and
\[
\lim_{n\to\infty} \frac{\log (t_n-1)} {\log s_n} =0.
\]
Then
\begin{eqnarray*}
\dim_HF(s,t)=\left(2+\limsup_{n\to \infty}\frac{\log s_{n+1}} {\log s_1s_2\cdots s_n}\right)^{-1}.
\end{eqnarray*}
\end{lem}
The proof of Lemma \ref{lemma-general} is essentially contained in the proof of the lower bound of the dimension of $\underline{E}(\psi)$ in Subsection \ref{Dim-E-under}. So the details are left for the reader. A special case of Lemma \ref{lemma-general} can be found in \cite[Lemma 2.3]{LR}.

The next lemma will be used to obtain the upper bound of the Hausdorff dimensions of $\overline{E}(\psi)$ and $\underline{E}(\psi)$.
\begin{lem}[\cite{Lu}, Main Theorem]\label{Luczak} 
For any $a>1, b>1$,
\begin{align*}
&\dim_H\left\{x\in [0,1]: a_n(x)\geq a^{b^n}, \forall n\geq 1\right\}\\
=&\dim_H\left\{x\in [0,1]: a_n(x)\geq
a^{b^n}, {\text{ for infinitely many } n}\right\}=\frac{1}{b+1}.
\end{align*} 
\end{lem}

\section{Proofs}
\subsection{Dimension of $\overline{E}(\psi)$}

We first calculate the Hausdorff dimension of $\overline{E}(\psi)$ as defined in (\ref{upper-set}).
We will only give the proof for $1<b< \infty$. The case $b=\infty$ can be obtained by a standard limit procedure.

Upper bound: For $x\in \overline{E}(\psi)$,
 let $S_n(x):= \log a_1(x) + \cdots + \log a_n(x) $.
  Then for any $\delta>0$, there are infinitely many $n$'s such that
  \begin{equation}\label{eqn:klddkl}
  S_n(x) \geq \psi(n)(1-\delta).
  \end{equation}
  For each such $n$ there exists an $i\leq n$ such that
  \begin{equation}\label{eqn:plplpl}
  \log a_i(x) \geq \frac{\psi(n)}{2n}(1-\delta).
  \end{equation}
  Moreover, for infinitely many $i$'s we can find $n\geq i$ satisfying \eqref{eqn:klddkl} such that \eqref{eqn:plplpl} holds. Indeed, if $i_0$ was the largest $i$ with this property, for every $n>i_0$ satisfying \eqref{eqn:klddkl} we would have
  \[
  S_n(x) = S_{i_0}(x) + \sum_{k=i_0+1}^n \log a_k(x) < S_{i_0}(x) + \frac {1-\delta} 2 \psi(n)
  \]
  and hence,
  \[
  \psi(n) < 2(1-\delta)^{-1} S_{i_0}(x).
  \]
Thus $\psi(n)$ for $n$ satisfying \eqref{eqn:klddkl} would be uniformly bounded. This contradicts the assumption $\psi(n)/n \to \infty$. 

  By the definition (\ref{def-bB}) of $b$, for any small $\epsilon >0$ ($\epsilon<(b-1)/100$ is enough) we have $\psi(n)>(b-\epsilon)^n>1$ and $(b-\epsilon)^n (1-\delta)/2n > (b-2\epsilon)^n$ for all $n$ large enough. Thus, by (\ref{eqn:plplpl}) there are infinitely many $i$'s, such that
  \[\log a_i(x)> \frac{(b-\epsilon)^n}{2n}(1-\delta)>(b-2{\epsilon})^n \geq (b-2{\epsilon})^i. \]
 It then follows that the set $\overline{E}(\psi)$ is included in the set
  \[
  \left\{x\in [0,1]: a_i(x) >e^{(b-2{\epsilon})^i}, \ {\text{ for infinitely many } n}\right\}.
  \]
By Lemma \ref{Luczak},
 the Hausdorff dimension of $\overline{E}(\psi)$ is bounded by $1/\big(1+(b-2\epsilon)\big)$ from above. Letting $\epsilon \to 0$, we obtain the upper bound.

 \medskip
 Lower bound: We will construct a Cantor type subset of $\overline{E}(\psi)$. 

We want to construct a sequence $\{c_n\}_{n\geq 1}$ of positive real numbers 
such that 
\begin{equation}\label{subset-Ec}
\limsup_{n\to\infty}\frac{\log c_1+\cdots+\log c_n}{\psi(n)}=1,
\end{equation}
and
\begin{equation}\label{dim-cal}
\limsup_{n\to\infty}\frac{\log c_{n+1}}{\log c_1+\cdots+\log
c_n}\leq b-1+\epsilon.
\end{equation}
The condition (\ref{subset-Ec}) will be used to construct points in $\overline{E}(\psi)$ while the condition (\ref{dim-cal}) will be helpful for estimating the Hausdorff dimension.

Let us begin by defining an auxiliary function
\[
\phi(n) = \min_{k\geq n} \psi(k).
\]
As $\psi(n)\to \infty$, $\phi(n)$ is well defined for all $n$. Obviously, $\phi\leq\psi$ and $\phi$ is a nondecreasing function. One can easily check that there exist infinitely many $n$'s for which $\phi(n)=\psi(n)$. One more property that we will soon use is
\begin{equation} \label{eqn:phi}
\phi(n) \neq \psi(n) \implies \phi(n) = \min_{k\geq n} \psi(k) = \min_{k>n} \psi(k) = \phi(n+1).
\end{equation}

Let $c_1=e^{\phi(1)}$ and $$
c_2=\min\left\{\frac{e^{\phi(2)}}{c_1}, \ c_1^{b-1+\epsilon}\right\}.
$$Assume that $c_{n-1}$ has already been well defined, then set
\begin{align}\label{def-cn}
c_{n}=\min\left\{\frac{e^{\phi(n)}}{\prod_{k=1}^{n-1} c_k}, \
\prod_{k=1}^{n-1} c_k^{b-1+\epsilon}\right\}.
\end{align} 
Observe that as $\phi$ is nondecreasing, $c_n\geq 1$ for all $n\geq 1$.

We can check that
\begin{equation*}
\limsup_{n\to\infty}\frac{\log c_{n+1}}{\log c_1+\cdots+\log
c_n}\leq \limsup_{n\to\infty}\frac{\log \left(\prod_{k=1}^n c_k^{b-1+\epsilon}\right) }{\log \prod_{k=1}^n c_k}=b-1+\epsilon.
\end{equation*}
Thus (\ref{dim-cal}) holds.

To prove \eqref{subset-Ec} we will need several steps.
By the definition (\ref{def-bB}) of $b$, we claim that 
\begin{align}\label{c_n-inf}
c_{n}=\frac{e^{\phi(n)}}{\prod_{k=1}^{n-1} c_k} \quad \text{ for infinitely many } n.
\end{align}
Indeed, if it is not true then for some $N$ and for all $n\geq N$ we have
\begin{equation} \label{eqn:kssk1}
c_{n}=\prod_{k=1}^{n-1} c_k^{b-1+\epsilon}
\end{equation}
and
\begin{equation} \label{eqn:kssk2}
e^{\phi(n)} > \prod_{k=1}^{n-1} c_k^{b+\epsilon}.
\end{equation}
The formula \eqref{eqn:kssk1} implies
\[
\prod_{k=1}^n c_k = \left(\prod_{k=1}^N c_k\right)^{(b+\epsilon)^{n-N}},
\]
which together with \eqref{eqn:kssk2} and $\phi \leq \psi$ is in contradiction with \eqref{def-bB}.

Observe now that by \eqref{eqn:phi} and \eqref{def-cn}, if the equality in \eqref{c_n-inf} holds for some $n$ such that $\phi(n) \neq \psi(n)$ then it will hold for $n+1$ as well (and $c_{n+1}=1$). Since $\phi(n)=\psi(n)$ infinitely often, repeating this argument until we get to some $n+k$ for which $\phi(n+k)=\psi(n+k)$, we prove
\begin{align}\label{c_n-inf+}
c_{n}=\frac{e^{\psi(n)}}{\prod_{k=1}^{n-1} c_k} \quad \text{ for infinitely many } n.
\end{align}

Denoting by $\{n_{j}\}$ the sequence of numbers satisfying \eqref{c_n-inf+}, we get
\begin{equation*}
\limsup_{n\to\infty}\frac{\log c_1+\cdots+\log c_n}{\psi(n)}\geq  \limsup_{j\to\infty}\frac{\log c_1+\cdots+\log c_{n_j}}{\psi(n_j)} =1.
\end{equation*}
 On the other hand, by \eqref{def-cn},
\begin{equation*}
\limsup_{n\to\infty}\frac{\log c_1+\cdots+\log c_n}{\psi(n)}\leq  \limsup_{n\to\infty}\frac{\log c_1+\cdots+\log c_n}{\phi(n)} \leq 1.
\end{equation*} 
Hence (\ref{subset-Ec}) holds.

Define
$$ E(\{c_n\}):=\{x\in [0,1): c_n \leq a_n(x)<2 c_n, \ {\rm{for \ all}}
 \ n\geq 1\}.
$$
Since $\psi(n)/n \to\infty$ as $n\to\infty$, by \eqref{subset-Ec}, $E(\{c_n\}) \subset\overline{E}(\psi)$.

Now we would like to apply Lemma \ref{lemma-FLWW} to estimate the Hausdorff dimension of $E(\{c_n\})$. 
However, observe that \eqref{subset-Ec} does only imply that $\limsup c_n = \infty$, but not that $c_n\to \infty$. For example, if $\psi$ (or $\phi$) has a long plateau then we can find many $n$'s with $c_n=1$, and hence for a function $\psi$ with infinitely many long plateaux there might exist a subsequence $c_{n_i} \equiv 1$.
 On the other hand, to apply Lemma \ref{lemma-FLWW} we need the condition $c_n\to \infty$ as $n\to \infty$. So, some modifications on the subset $E(\{c_n\})$ are needed.

 By the condition that $\psi(n)/n\to \infty$
as $n\to\infty$, we can choose an increasing sequence $\{n_k\}_{k=1}^{\infty}$
such that for each $k\geq 1$ 
$$ \frac{\psi(n)}{n}\geq k^2, \
{\rm{when}}\ n\geq n_k.
$$
Take $\alpha_n=2$ if $1\leq n<n_1$ and
$$  \alpha_n=k+1, \ {\rm{when}} \ n_k\leq n<n_{k+1}.
$$
Let $k(n)$ be such that $n_{k(n)} \leq n < n_{k(n)+1}$.
Then
\[
\lim_{n\to\infty}\frac{\log \alpha_1+\cdots+\log
\alpha_n}{\psi(n)}\leq \lim_{n\to\infty}\frac {n \cdot (k(n)+1)} {n \cdot k(n)^2} =0
\]
and
\[\lim_{n\to\infty}\frac{\log
\alpha_{n+1}}{\log \alpha_1+\cdots+\log \alpha_n}\leq \lim_{n\to\infty}\frac {\log (n+1)} {n \cdot \log 2}  =0.
\]

 Since $c_n\geq 1$ and $\alpha_n\geq 2$ for all $n\geq 1$,
we have $$\log c_n\leq \log ( c_n+\alpha_n )\leq \log c_n+2\log
\alpha_n \quad \forall n\geq 1.$$
So, by taking $s_n= c_n+\alpha_n $ for each $n\geq 1$,
we get $s_n\to\infty$ as $n\to
\infty$ and
$$
\limsup_{n\to\infty}\frac{\log s_1+\cdots+\log
s_n}{\psi(n)}=1.
$$
Define $$ E(\{s_n\}):=\{x\in [0,1):s_n\leq a_n(x)<2s_n, \ {\rm{for \
all}}
 \ n\geq 1\}.
$$ Then $$E(\{s_n\})\subset\overline{E}(\psi).$$ As $s_n\to\infty$ as $n\to
\infty$, by Lemma \ref{lemma-FLWW}, we have $$ \dim_HE(\{s_n\})=\left(2+\limsup_{n\to \infty}\frac{\log s_{n+1}} {\log s_1+\cdots+\log s_n}\right)^{-1}.
$$
Note that \begin{eqnarray*}&& \limsup_{n\to\infty}\frac{\log
s_{n+1}}{\log s_1+\cdots+\log
s_n}\\&=&\limsup_{n\to\infty}\frac{\log
 (c_{n+1}+\alpha_{n+1}) }{\log  (c_1+\alpha_1)+\cdots+\log (c_n+\alpha_n)}\\
&\leq &\limsup_{n\to\infty}\frac{\log c_{n+1}+2\log\alpha_{n+1}}{\log
 (c_1+\alpha_1)+\cdots+\log (c_n+\alpha_n)}\\
&\leq &\limsup_{n\to\infty}\frac{\log c_{n+1}}{\log c_1+\cdots+\log
c_n}+\limsup_{n\to\infty}\frac{2\log\alpha_{n+1}}{\log
\alpha_1+\cdots+\log \alpha_n}\\
&\leq &b-1+\epsilon.
\end{eqnarray*}Hence, $$
\dim_H\overline{E}(\psi)\geq \dim_HE(\{s_n\})\geq \frac{1}{b+1+\epsilon}.$$


\medskip
\subsection{Dimension of $\underline{E}(\psi)$}\label{Dim-E-under}
As in the calculation of the Hausdorff dimension of $\overline{E}(\psi)$, we will only give the proof for $1<B< \infty$ and the easy case $B=\infty$ is left for the reader.

Upper bound: By the definition (\ref{def-bB}) of $B$, for any $\epsilon>0$, there is a sequence $\{n_i\}$ such that
\[
  \psi(n_i)> (B-\epsilon)^{n_i}.
\]
Denoting $S_n(x)= \log a_1(x) + \cdots + \log a_n(x)$, for all $x\in \underline{E}(\psi)$, for any $\delta>0$, we have
$$S_n(x) \geq \psi(n)(1-\delta), \ \forall n\geq 1.$$
Thus
\[S_{n_i}(x) \geq (B-\epsilon)^{n_i}(1-\delta).\]
Then there exists $j\leq n_i$ such that
\begin{equation} \label{eqn:somet}
\log a_{j}(x) \geq (B-\epsilon)^{n_i}(1-\delta)/2n_i > (B-2\epsilon)^{j}.
\end{equation}
 By the same argument as in the previous subsection, we have infinitely many such $j$'s for which we can find $n_i$ satisfying \eqref{eqn:somet}.  Thus by Lemma \ref{Luczak}, the Hausdorff dimension of $\underline{E}(\psi)$ is bounded by $1/(1+(B-2\epsilon))$ from above.
 The upper bound then follows by letting $\epsilon \to 0$.

 \medskip
 Lower bound: As in the previous subsection, a Cantor type subset of $\underline{E}(\psi)$ will be constructed.

For any $\epsilon>0$, define
\begin{align}\label{def-Ai}
A_i= \sup_{n\geq i}\exp\{\psi(n)(B+\epsilon)^{i-n}\}.
\end{align}
Since $\limsup_{n\to\infty}\frac{\log \psi(n)}{n}=\log B$, we have for $n$ large enough
\[
\psi(n)\leq (B+\epsilon/2)^n.
\]
This implies
\[
\psi(n)(B+\epsilon)^{i-n} \leq (B+\epsilon/2)^n  (B+\epsilon)^{i-n} \to 0 \quad (n\to\infty).
\]
Hence in the definition (\ref{def-Ai}) of $A_i$ the supremum is achieved.

Denote by $t_i\geq i$ the smallest number for which $A_i = \exp\{\psi(t_i)(B+\epsilon)^{i-t_i}\}$.
Figure \ref{fig} shows a way to find $t_i$ and $A_i$.


We remark that such defined $A_i$ is the smallest function satisfying 
\begin{equation}\label{prop-A}
A_{i+1} \leq A_i^{B+\epsilon} \quad \text{and} \quad A_i\geq e^{\psi(i)}.
\end{equation}
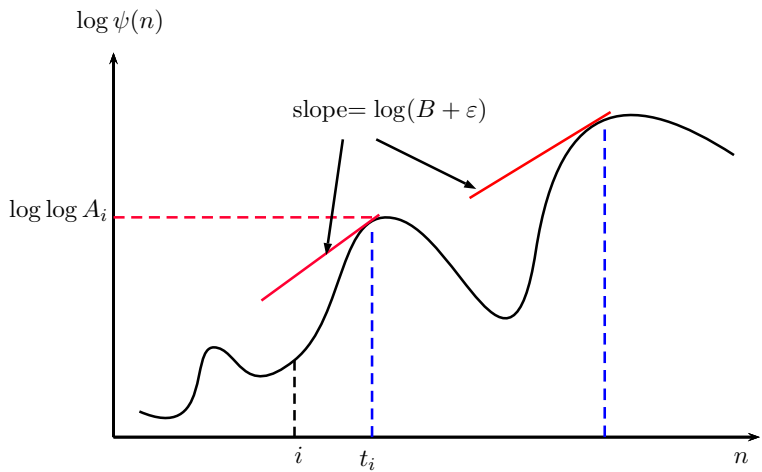
\begin{figure}
\psscalebox{0.85 0.85} 
{
\begin{pspicture}(0,-3.5759218)(11.479986,3.5759218)
\definecolor{colour0}{rgb}{1.0,0.0,0.2}
\rput(1.6129032,-3.0597925){\psaxes[linecolor=black, linewidth=0.04, tickstyle=full, axesstyle=axes, labels=none, ticks=none, dx=1.0cm, dy=1.0cm]{->}(0,0)(0,0)(10,6)}
\psbezier[linecolor=black, linewidth=0.04](2.0129032,-2.6597927)(3.1629033,-3.1597927)(2.7795699,-1.6597927)(3.1629033,-1.6597927)(3.5462365,-1.6597927)(3.600238,-2.5222912)(4.4129033,-1.8597926)(5.2255683,-1.197294)(5.0773907,0.61137044)(6.012903,0.34020737)(6.9484158,0.0690443)(7.762903,-2.6597927)(8.146236,-0.15979263)(8.52957,2.3402073)(9.67957,2.3402073)(11.212903,1.3402073)
\psline[linecolor=red, linewidth=0.04](9.303715,2.0111241)(7.1220913,0.6692905)
\psline[linecolor=colour0, linewidth=0.04](5.724642,0.4100824)(3.9011645,-0.92966765)
\psline[linecolor=blue, linewidth=0.04, linestyle=dashed, dash=0.17638889cm 0.10583334cm](5.612903,0.14020738)(5.612903,-3.0597925)
\psline[linecolor=blue, linewidth=0.04, linestyle=dashed, dash=0.17638889cm 0.10583334cm](9.212903,1.7402073)(9.212903,-3.0597925)(9.212903,-3.0597925)
\psline[linecolor=black, linewidth=0.04, linestyle=dashed, dash=0.17638889cm 0.10583334cm](4.4129033,-1.8597926)(4.4129033,-3.0597925)(4.4129033,-3.0597925)
\rput[bl](11.212903,-3.4597926){$n$}
\rput[bl](4.4129033,-3.4597926){$i$}
\psline[linecolor=black, linewidth=0.04](1.6129032,0.3687788)(1.6129032,0.3687788)
\psline[linecolor=colour0, linewidth=0.04, linestyle=dashed, dash=0.17638889cm 0.10583334cm](5.612903,0.3687788)(1.6129032,0.3687788)(1.6129032,0.3687788)
\rput[bl](1.0414747,3.2259216){$\log \psi(n)$}
\rput[bl](0.0,0.2950461){$\log\log A_i$}
\rput[bl](4.387097,1.8434331){slope$=\log(B+\epsilon)$}
\psline[linecolor=black, linewidth=0.04, arrowsize=0.05291666666666667cm 2.0,arrowlength=1.4,arrowinset=0.0]{->}(5.16129,1.5853686)(4.903226,-0.22108296)
\psline[linecolor=black, linewidth=0.04, arrowsize=0.05291666666666667cm 2.0,arrowlength=1.4,arrowinset=0.0]{->}(5.677419,1.5853686)(7.225806,0.8111751)
\rput[bl](5.419355,-3.5759218){$t_i $}
\end{pspicture}
}
\caption{Finding $A_i$.}\label{fig}
\end{figure}


Let
\[
Z:=\liminf_{n\to\infty}\frac{\sum_{i=1}^n\log A_i}{\psi(n)}.
\]
By (\ref{prop-A}),  we have
\[
Z\geq \liminf_{n\to\infty}\frac{\log A_n}{\psi(n)}\geq 1.
\]

We will use these $A_i$ and $Z$ to construct our Cantor type subset. Before the construction, let us prove the following proposition. 
\begin{prop}\label{prop-finite}
We have $Z<\infty$.
\end{prop}
\begin{proof}
Remark that for many consecutive $i$'s the $t_i$ will be the same. More precisely, $t_{i}=t_{i+1}=\cdots=t_{t_i}$.  Let $\{\ell_j\}$ be the sequence of all $t_i$'s in the increasing order, without repetitions. 
Notice that for these $\ell_j$, we have
\[
\log A_{\ell_j}=\psi(\ell_j),
\]
and for $k\in (\ell_{j-1}, \ell_j]$,
\[
\log A_k= \psi(\ell_j)(B+\epsilon)^{k-\ell_j}.
\]
Thus for $0<\epsilon <B-1$
\begin{equation} \label{eqn:lsssl}
\sum_{k=\ell_{j-1}+1}^{\ell_j} \log A_k =\sum_{k=\ell_{j-1}+1}^{\ell_j}\psi(\ell_j)(B+\epsilon)^{k-\ell_j} \leq \frac{B+\epsilon}{B+\epsilon-1} \cdot \psi(\ell_j).
\end{equation}

Denote $S_n\psi:= \sum_{k=1}^{n}\psi(k)$. Proposition \ref{prop-finite} follows directly from the following two lemmas.
\begin{lem}\label{lem-finite}
We have
\[
\liminf_{n\to\infty} \frac{S_n\psi}{\psi(n)}<\infty.
\]
\end{lem}
\begin{proof}
For $\epsilon>0$, we  will show that there exist infinitely many $n$, such that $\psi(n)> \epsilon S_{n-1}\psi$. If this was not true, then for all large $n$ we would have
\[
S_n\psi = S_{n-1}\psi +\psi(n) \leq (1+\epsilon)S_{n-1}\psi.
\]
Thus by Stolz-Ces\`aro Theorem
\[
\limsup_{n\to\infty} \frac{\log S_n\psi}{n} \leq \limsup_{n\to\infty} \frac{\log S_n\psi-\log S_{n-1}\psi}{n-(n-1)} \leq \log (1+\epsilon),
\]
which is impossible since we have
\[
\limsup_{n\to\infty} \frac{\log S_n\psi}{n} \geq\limsup_{n\to\infty} \frac{\log \psi(n)}{n}=B>\log (1+\epsilon).
\]
Denote by $\{n_j\}$ a sequence such that $\psi(n_j)> \epsilon S_{n_j-1}\psi$. Then
\[
\frac{S_{n_j}\psi}{\psi(n_j)}=\frac{S_{n_j-1}\psi +\psi(n_j)}{\psi(n_j)} \leq 1+\frac{1}{\epsilon}<\infty,\]
and the conclusion follows.
\end{proof}

\begin{lem}\label{lem-subseq}
The following limit is finite:
\[
\liminf_{j\to\infty} \frac{S_{\ell_j}\psi}{\psi(\ell_j)}<\infty.
\]
\end{lem}
\begin{proof}
Denote $L:=\liminf_{n\to\infty} \frac{S_n\psi}{\psi(n)}$.
Fix $0<\epsilon <B-1$ and let $\{m_k\}$ be the sequence such that 
\[
\frac{S_{m_k}\psi}{\psi(m_k)} \leq L+\epsilon.
\]
Each $m_k$ is in some $(\ell_{j-1}, \ell_j]$, hence for this $j$
\begin{align*}
S_{\ell_j}\psi =S_{m_k}\psi+ \sum_{i=m_k+1}^{\ell_j}\psi(i)&\leq (L+\epsilon) \psi(m_k) +\sum_{i=m_k+1}^{\ell_j}\psi(i)\\
&\leq (L+\epsilon) \sum_{i=m_k}^{\ell_j}\psi(i).
\end{align*}
By definition of $\ell_j$, for $i\in [m_k, \ell_j] \subset (\ell_{j-1}, \ell_j]$
\[
\psi(i) \leq \frac{1}{(B+\epsilon)^{\ell_j-i}}\psi(\ell_j),
\]
Thus
\[
S_{\ell_j}\psi  \leq \frac{B+\epsilon}{B+\epsilon-1} \cdot (L+\epsilon) \psi(\ell_j).
\]
The result then follows.
\end{proof}
By \eqref{eqn:lsssl}
\[
\sum_{k=1}^{\ell_j} \log A_k \leq C \cdot S_{\ell_j}\psi.
\]
Hence
\[
Z \leq \liminf_{j\to\infty} \frac {\sum_{k=1}^{\ell_j} \log A_k} {\psi(\ell_j)} \leq C \frac {S_{\ell_j}\psi} {\psi(\ell_j)} < \infty,
\]
which completes the proof of Proposition \ref{prop-finite}
\end{proof}

\medskip
Now we continue the construction of the subset and the estimation of the lower bound.
%

Choose a sequence $\epsilon_i\to 0$ such that
\begin{align}\label{assum-1}
 \lim_{n\to\infty}\frac{\sum_{i=1}^n \log (1\pm \epsilon_i)}{\psi(n)}=0,
 \end{align}
 \begin{align}\label{assum-2}
\lim_{n\to\infty} \frac{|\sum_{i=1}^n \log (2\epsilon_i)|}  {\log A_{n+1}}=0,
 \end{align}
 and
\begin{align}\label{assum-3}
\text{for each $i$} \quad  W_i\cap \mathbb{N}\neq \emptyset,
 \end{align}
where \begin{align*}
 W_i:=[A_i^{1/Z}(1-\epsilon_i), \ A_i^{1/Z}(1+\epsilon_i)].
 \end{align*}

Observe that the condition (\ref{assum-2}) implies the condition (\ref{assum-3}) for all $i$ large enough. Observe also that for large $i$, (\ref{assum-2}) and (\ref{assum-3}) say only that $\epsilon_i$ is greater than some negative power of $A_i$ (which is increasing superexponentially fast) while the condition (\ref{assum-1}) is satisfied for all decreasing sequences, so there is no problem to find such a sequence $\{\epsilon_i\}$ satisfying all the conditions (\ref{assum-1})-(\ref{assum-3}).

Denote by $E$ the set of numbers $x$ such that for all $i$, $a_i(x)$ is in the interval $W_i$. By (\ref{assum-3}), $E$ is nonempty.
By (\ref{assum-1}), for all $x\in E$
\[
   \liminf_{n\to\infty} \frac{\sum_{j=1}^n\log a_j(x)}{\psi(n)}= \liminf_{n\to\infty} \frac{\frac{1}{Z}\sum_{j=1}^n\log A_j}{\psi(n)}=1.
\]
So $E \subset \underline{E}(\psi)$. 

Now we will estimate the Hausdorff dimension of $E$ from below. To this end, we define a probability measure $\mu$ on $E$. For each position, we distribute the probability evenly.
That is, for each possible $a_i$, we give the probability
\[
p_i =\frac{1}{1+|[\lceil A_i^{1/Z}(1-\epsilon_i)\rceil, \lfloor A_i^{1/Z}(1+\epsilon_i)\rfloor]|}\approx \frac{1}{2\epsilon_i A_i^{1/Z}}.
\]
Here and in what follows, we follow \cite{JR}. For simplicity we give only the main term of the calculations.

By the probability distribution, for each basic interval $I_n=I_n(a_1,\dots a_n)$, we have
\[
\mu(I_n)=\prod_{i=1}^n p_i \approx \prod_{i=1}^n (2\epsilon_i A_i^{1/Z})^{-1}.
\]
By (\ref{length-int}),
\[
|I_n| \approx \prod_{i-1}^n (A_i^{1/Z})^{-2}.
\]
To calculate the local dimension of $x\in E$ we will use a smaller interval $D_n$ included in $I_n$:
\[
  D_n=\cup_{a_{n+1}\geq A_{n+1}^{1/Z}(1-\epsilon_{n+1})} I_{n+1}(a_1, \cdots, a_n a_{n+1}).
\]
Since $a_i\in [A_i^{1/Z}(1-\epsilon_i), \ A_i^{1/Z}(1+\epsilon_i)]$ and $A_i$ grows super-exponentially, the Hausdorff dimension will be determined by calculating the local dimension
\[
\liminf_{n\to\infty} \frac{\log \mu(D_n)}{\log |D_n|}.
\]
(See Section 4 of Jordan and Rams \cite{JR}.)

On the one hand,
\[
-\log \mu(D_n)=-\log \mu(I_n) =\sum_{i=1}^n \log (2\epsilon_i) + \frac{1}{Z}\sum_{i=1}^n \log A_i.
\]
By the property that $A_{i+1}\leq A_i^{B+\epsilon}$, we deduce that
for big $n$,
\begin{align}\label{mu-D_n}
\begin{split}
-\log \mu(D_n) &\geq \sum_{i=1}^n \log (2\epsilon_i)+ \frac{1}{Z} \sum_{i=1}^n \frac{1}{(B+\epsilon)^i} \log A_{n+1} \\
&\approx \frac{1}{Z(B+\epsilon-1)} \log A_{n+1}.
\end{split}
\end{align}
The last $\approx$ follows from (\ref{assum-2}).

On the other hand, the length of the interval $D_n$ is
\[
|D_n|\approx |I_n|\cdot A_{n+1}^{-1/Z}.
\]
Thus 
\[
-\log |D_n|\approx -\log |I_n| + \frac{1}{Z}\log A_{n+1} \approx -2 \log \mu(D_n) +\frac{1}{Z}\log A_{n+1}.
\]
Hence 
\begin{eqnarray*}
\frac{-\log \mu(D_n)}{-\log |D_n|} &\approx& \frac{-\log \mu(D_n)}{  -2 \log \mu(D_n) +\frac{1}{Z}\log A_{n+1}} = \frac{1}{2+\frac{\frac{1}{Z} \log A_{n+1}}{-\log \mu(D_n)}}.
\end{eqnarray*}
Then by (\ref{mu-D_n}), we obtain
\begin{eqnarray*}
\liminf_{n\to\infty}\frac{-\log \mu(D_n)}{-\log |D_n|} \geq \frac{1}{2+(B+\epsilon -1)}=\frac{1}{B+1+\epsilon}.
\end{eqnarray*}
Therefore the lower bound follows from the Frostman Lemma (see \cite[Principle 4.2]{Fa1}).

\bigskip
\noindent{\bf Acknowledgements.}
The authors thank Bao-Wei Wang for the fruitful discussions.
L. Liao was partially supported by the ANR, grant 12R03191A -MUTADIS (France).
M.Rams was partially supported by the MNiSW grant N201 607640 (Poland).

 {\small

\end{document}